\documentclass[11pt]{article}
\usepackage{layout}
\usepackage[makeroom]{cancel}
\usepackage{bbm}

\usepackage{amsfonts,dsfont}
\usepackage{amsmath}
\usepackage{array}
\usepackage{amssymb,bm}
\usepackage{amsthm}
\usepackage{graphicx} 
\usepackage{mathtools}
\usepackage{multicol}
\usepackage{mathrsfs}
\usepackage{enumerate,enumitem}
\usepackage[rgb,dvipsnames]{xcolor}
\usepackage{float}
\usepackage[normalem]{ulem}
\usepackage{circuitikz}
\usepackage{cite}
\usepackage[colorlinks,linkcolor=blue,citecolor=blue,pagebackref,hypertexnames=false, breaklinks]{hyperref}
\usepackage{float}

\newtheorem{theorem}{Theorem}[section]

\newtheorem{definition}[theorem]{Definition}
\newtheorem{proposition}[theorem]{Proposition}

\newtheorem{example}[theorem]{Example}
\newtheorem{examples}[theorem]{Examples}

\newtheorem{open question}[theorem]{Open Question}
\newtheorem{c/p}[theorem]{Conjecture/Proposition}

\newcommand{\ang}[1]{\left<#1\right>} 

\def\vint{\mathop{\mathchoice%
 {\setbox0\hbox{$\displaystyle\intop$}\kern 0.22\wd0%
 \vcenter{\hrule width 0.6\wd0}\kern -0.82\wd0}%
 {\setbox0\hbox{$\textstyle\intop$}\kern 0.2\wd0%
 \vcenter{\hrule width 0.6\wd0}\kern -0.8\wd0}%
 {\setbox0\hbox{$\scriptstyle\intop$}\kern 0.2\wd0%
 \vcenter{\hrule width 0.6\wd0}\kern -0.8\wd0}%
 {\setbox0\hbox{$\scriptscriptstyle\intop$}\kern 0.2\wd0%
 \vcenter{\hrule width 0.6\wd0}\kern -0.8\wd0}}%
 \mathopen{}\int}

\DeclareMathOperator\arctanh{arctanh}

\newcommand{\R}{\mathbb R}

\newcommand{\M}{\mathbb M}

\newcommand{\bS}{\mathbb S}

\newcommand{\be}{\beta}

\DeclareMathOperator{\Cut}{Cut}

\usepackage[textwidth=3.2cm]{todonotes}

\usepackage{geometry}
\title{Brownian motions and heat kernel lower bounds on  K\"ahler and quaternion K\"ahler manifolds}

\author{ Fabrice Baudoin\footnote{Partly supported by the NSF grant DMS~1901315.}, Guang Yang}

\date{\today}

\begin{document}

\maketitle

\begin{abstract}
We study the radial parts of the Brownian motions on K\"ahler and quaternion K\"ahler manifolds. Thanks to sharp Laplacian comparison theorems, we deduce  as a consequence a sharp Cheeger-Yau type lower bound for the heat kernels of such manifolds and also sharp Cheng's type estimates for the Dirichlet eigenvalues of metric balls.  
\end{abstract}

\tableofcontents

\section{Introduction}

It is by now well established that on Riemannian manifolds the study of the radial parts of the Brownian motions allows to prove the  sharp Cheeger-Yau  lower bound \cite{MR615626} for the heat kernel, and as a consequence the sharp Cheng's estimate \cite{MR378001} for the eigenvalues of metric balls, see the paper \cite{ichihara} and the book \cite{MR1882015}. Those methods were then extended in the framework of RCD spaces in \cite{MR3950013} and adapted to sub-Riemannian manifolds in \cite{baudoin2020radial}. The goal of the present paper is to use similar probabilistic techniques to prove a sharp Cheeger-Yau  heat kernel lower bound on K\"ahler and quaternion K\"ahler manifolds. In K\"ahler manifold such techniques are available due to a recent Laplacian comparison theorem  proved Ni-Zheng \cite{MR3858834}. In  quaternion K\"ahler manifolds, we prove a sharp Laplacian comparison theorem that allows us to apply those techniques.  Concerning the sharp lower bounds for the heat kernels, our results are then the following.

In K\"ahler manifolds  we obtain:

\begin{theorem}[Cheeger-Yau estimate on K\"ahler manifolds, See Theorem \ref{comparison heat kernel}]
Let $\M$ be a K\"ahler manifold. Assume that $H \ge 4k$ and that $\mathrm{Ric}^\perp \ge (2m-2)k$ for some $ k \in \mathbb{R}$, where $H$ denotes the holomorphic sectional curvature and $\mathrm{Ric}^\perp$ the orthogonal Ricci curvature. Then, denoting by $p^R_t(x,y)$ the Dirichlet  heat kernel of $\mathbb{M}$ on a metric ball of radius $R>0$ one has for every $t >0$ and $x,y$ inside of the ball,
\[
p_t (x,y) \ge p_t^{k,R} ( 0, d(x,y)) 
\]
where  $p_t^{k,R} $ is the Dirichlet heat kernel of a metric ball of radius $R$ in the K\"ahler model of holomorphic sectional curvature  $4k$.
\end{theorem}

The K\"ahler model for $k=0$ is the complex flat space $\mathbb{C}^m$, for $k=1$ it is the complex projective space $\mathbb{C}P^m$ and for $k=-1$, it is the complex hyperbolic space $\mathbb{C}H^m$.


In quaternion K\"ahler manifolds, we obtain:

\begin{theorem}[Cheeger-Yau estimate on quaternion K\"ahler manifolds, See Theorem \ref{comparison heat kernel 2}]
Let $\M$ be a quaternion K\"ahler manifold. Assume that  $Q \ge 12k$ and that $\mathrm{Ric}^\perp \ge (4m-4)k$ for some $ k \in \mathbb{R}$, where $Q$ denotes the quaternionic sectional curvature and $\mathrm{Ric}^\perp$ the orthogonal Ricci curvature. Then, denoting by $p^R_t(x,y)$ the Dirichlet  heat kernel of $\mathbb{M}$ on a metric ball of radius $R>0$ one has for every $t >0$ and $x,y$ inside of the ball,
\[
p_t (x,y) \ge q_t^{k,R} ( 0, d(x,y)) 
\]
where  $q_t^{k,R} $ is the Dirichlet heat kernel of a metric ball of radius $R$ in the quaternion K\"ahler model of quaternionic sectional curvature  $12k$.
\end{theorem}

The quaternion K\"ahler model for $k=0$ is the quaternionic flat space $\mathbb{H}^m$, for $k=1$ it is the quaternionic projective space $\mathbb{H}P^m$ and for $k=-1$, it is the quaternionic hyperbolic space $\mathbb{H}H^m$.

We note that since K\"ahler or quaternionic K\"ahler manifolds are Riemannian manifolds, the classical Cheeger-Yau lower bound \cite{MR615626} is available. However the Riemannian model spaces spheres and hyperbolic spaces are not K\"ahler or quaternionic K\"ahler models (except for $m=1$), therefore the two above theorems are sharper.

\

The paper is organized as follows. In section 2, we introduce the basic definitions and notations used throughout the paper. We also study the Brownian motions on the K\"ahler and quaternion K\"ahler models. Such study is important, since those Brownian motions provide the model processes with respect to which we aim to develop a comparison theory. In particular,   the radial parts of those Brownian motions are one-dimensional diffusions whose generators can explicitly be computed. A summary of those generators is given in section \ref{summary}. In section 3 we establish sharp Laplacian comparison theorems on K\"ahler and quaternionic K\"ahler manifolds. The K\"ahler case is known and due to Ni-Zheng \cite{MR3858834}. We give a slightly different and self-contained proof which is easy to adapt to the quaternion K\"ahler case. The quaternion K\"ahler case is new. Both of those Laplacian comparison theorems are sharp in the sense that we obtain an equality for the model spaces. Section 4 is devoted to the proof of the comparison theorems. Using the approach  by Ichihara \cite{ichihara} we prove, thanks to the results proved in the previous sections, the sharp Cheeger-Yau lower bounds for the heat kernels. As an easy consequence we deduce a sharp Cheng's type estimate for the first eigenvalue of metric balls.


\section{Brownian motion on K\"ahler and quaternion K\"ahler model manifolds}

In this section we fix notations and give some reminders about K\"ahler and quaternion K\"ahler manifolds and study the Brownian motions on the model spaces of those geometries. Brownian motions on  K\"ahler models and quaternion K\"ahler models have already been studied in disparate places in the literature, so that the present  section is essentially a survey of known results. However, our goal is a unified  presentation which has interest on its own. We refer to \cite{BDW2, BDW1,MR3719061} and the references therein for further details.

\subsection{Basic definitions}

K\"ahler and quaternion K\"ahler manifolds are Riemannian manifolds equipped with some  invariant  $(1,1)$ tensors preserving the metric and inducing a complex or quaternionic structure.  In this paper, we will take the point of view of real Riemannian geometry to study those structures. A detailed presentation of this viewpoint about K\"ahler and quaternion K\"ahler manifolds is given in Chapter 2 and Chapter 14 of the book by Besse \cite{besse} to which we refer for further references.

Throughout the paper, let $(\M,g)$ be a smooth complete Riemannian manifold. Denote by $\nabla$ the Levi-Civita connection on $\M$. 

\subsubsection{K\"ahler manifolds}

\begin{definition}
The manifold $(\M,g)$ is called a K\"ahler manifold, if there exists a smooth  $(1,1)$ tensor $J$ on $\M$ that satisfies:
\begin{itemize}
\item For every $x \in \M$, and $X,Y \in T_x\M$, $g_x(J_x X,Y)=-g_x(X,J_xY)$;
\item For every $x \in \M$, $J_x^2=-\mathbf{Id}_{T_x\M} $;
\item $\nabla J$=0.
\end{itemize}
The map $J$ is called a complex structure.
\end{definition}

On K\"ahler manifolds, we will be considering the following type of curvatures. Let
\[
R(X,Y,Z,W)=g ( (\nabla_X \nabla_Y -\nabla_Y \nabla_X -\nabla_{[X,Y]} )Z, W)
\]
be the Riemannian curvature tensor of $(\M,g)$. The holomorphic sectional curvature of the K\"ahler manifold $(\M,g,J)$ is defined as
\[
H(X)=\frac{R(X,JX,JX,X)}{g(X,X)^2}.
\]

The orthogonal Ricci curvature (see \cite{MR3996490}) of the K\"ahler manifold $(\M,g,J)$ is defined for a vector field  $X$ such that $g(X,X)=1$ by
\[
\mathrm{Ric}^\perp (X,X)=\mathrm{Ric} (X,X)-H(X),
\]
where $\mathrm{Ric}$ is the usual Riemannian Ricci tensor of $(\M,g)$.

\subsubsection{Quaternion K\"ahler manifolds}

In the paper we shall use the following definition of quaternion K\"ahler manifold, see Chapter 14 in \cite{besse}.

\begin{definition}
The manifold $(\M,g)$ is called a  quaternion K\"ahler manifold, if   there exists a covering of $\M$ by open sets $U_i$ and, for each $i$,  3 smooth  $(1,1)$ tensors $I,J,K$ on $U_i$ such that:

\begin{itemize}
\item For every $x \in U_i$, and $X,Y \in T_x\M$, $g_x(I_x X,Y)=-g_x(X,I_xY)$,  $g_x(J_x X,Y)=-g_x(X,J_xY)$, $g_x(K_x X,Y)=-g_x(X,K_xY)$ ;
\item For every $x \in U_i$, $I_x^2=J_x^2=K_x^2=I_xJ_xK_x=-\mathbf{Id}_{T_x\M} $;
\item For every $x \in U_i$, and $X\in T_x\M$  $\nabla_X I, \nabla_X J, \nabla_X K \in \mathbf{span} \{ I,J,K\}$;
\item For every $x \in U_i \cap U_j$, the vector space of endomorphisms of $T_x\M$ generated by $I_x,J_x,K_x$ is the same for $i$ and $j$.
\end{itemize}
\end{definition}

It is worth noting that in some cases like the quaternionic projective spaces for topological reasons the tensors $I,J,K$ may not be defined globally. However $\mathbf{span} \{ I,J,K\}$ may always be defined globally according to the last bullet point.

On quaternion K\"ahler manifolds, we will be considering the following curvatures. As above, let
\[
R(X,Y,Z,W)=g ( (\nabla_X \nabla_Y -\nabla_Y \nabla_X -\nabla_{[X,Y]} )Z, W)
\]
be the Riemannian curvature tensor of $(\M,g)$. We define the quaternionic sectional  curvature of the quaternionic K\"ahler manifold $(\M,g,J)$ as
\[
Q(X)=\frac{R(X,IX,IX,X)+R(X,JX,JX,X)+R(X,KX,KX,X)}{g(X,X)^2}.
\]

We define the orthogonal Ricci curvature  of the quaternionic K\"ahler manifold $(\M,g,I,J,K)$  for a vector field  $X$ such that $g(X,X)=1$ by
\[
\mathrm{Ric}^\perp (X,X)=\mathrm{Ric} (X,X)-Q(X),
\]
where $\mathrm{Ric}$ is the usual Riemannian Ricci tensor of $(\M,g)$.

\subsection{Model spaces and their Brownian motions}

The constant curvature model spaces of Riemannian geometry are the Euclidean spaces, the spheres and the hyperbolic spaces. Euclidean spaces are K\"ahler if the dimension is even and quaternion K\"ahler if the dimension is a multiple of 4. The only spheres and  hyperbolic spaces which are K\"ahler are the two dimensional ones. The only spheres and  hyperbolic spaces which are quaternion K\"ahler are the four dimensional ones.  In order to develop a comparison geometry for the Brownian motion in higher dimensional K\"ahler or quaternion K\"ahler geometry, one therefore needs to first study the Brownian motion on the models of those geometries. In this section, we review the K\"ahler and quaternion K\"ahler model spaces and their Brownian motions. All of those model spaces are rank one Riemannian symmetric spaces; As such, see \cite{B02},  the radial parts of the Brownian motions are diffusion processes.

\subsubsection{K\"ahler models}

\textbf{Flat model.}  The flat model of a K\"ahler manifold is 
\[
\mathbb{C}^m =\left\{ (z_1,\cdots,z_m), \, z_1,\cdots , z_m \in \mathbb{C} \right\}
\]
 equipped with its standard Hermitian inner product. The complex structure $J$ in that case is  just the component-wise multiplication by $i$. The Brownian motion $(W_t)_{t \ge 0}$ on  $\mathbb{C}^m$ is the diffusion process associated with the Laplace operator
\[
\Delta_{\mathbb{C}^m}=4\sum_{i=1}^m \frac{\partial^2}{\partial z_i \partial \bar{z}_i }  =\sum_{i=1}^m   \frac{\partial^2}{\partial^2 x_i} +\frac{\partial^2}{\partial^2 y_i} 
\]
where $x_i$ is the real part of $z_i$, $y_i$ its imaginary part and
\[
\frac{\partial }{\partial z_i }= \frac{1}{2} \left(  \frac{\partial }{\partial x_i } -i \frac{\partial }{\partial y_i }\right) , \quad \frac{\partial }{\partial \bar{z}_i }= \frac{1}{2} \left(  \frac{\partial }{\partial x_i } +i \frac{\partial }{\partial y_i }\right).
\]
One has
\[
W_t =\left( Z^1_t,\cdots,Z^m \right),
\]
where the $Z^i$'s are independent complex Brownian motions on $\mathbb C$. The radial part of $W$ defined by 
\[
r_t=| W_t|=\sqrt{ \sum_{i=1}^m | Z^i_t|^2}
\]
is itself a diffusion process with Bessel generator
\[
L_{\mathbb{C}^m}=\frac{\partial^2}{\partial r^2} + \frac{2m-1}{r} \frac{\partial}{\partial r} .
\]
We note that the radial part of the Lebesgue measure on $\mathbb{C}^m$ then writes:
\[
d\mu_{\mathbb{C}^m} = 2 \frac{\pi^m}{(m-1)!} \, r^{2m-1} dr, \quad r \ge 0.
\]

\

\textbf{Positively curved model.} The positively curved model of a K\"ahler manifold is the complex projective space $\mathbb CP^m$. It can be constructed as follows. Consider the unit sphere
\[
\bS^{2m+1}=\lbrace z=(z_1,\cdots,z_{m+1})\in \mathbb{C}^{m+1}, \| z \| =1\rbrace.
\]
There is an isometric group action of $\mathbb{S}^1=\mathbf{U}(1)$ on $\bS^{2m+1}$ which is  defined by $$e^{i\theta}\cdot(z_1,\cdots, z_{m+1}) = (e^{i\theta} z_1,\cdots, e^{i\theta} z_{m+1}). $$

The quotient space $\bS^{2m+1} / \mathbf{U}(1)$ is defined as $\mathbb{C}P^m$ and the projection map $\pi :  \bS^{2m+1} \to \mathbb{C}P^m$ is a Riemannian submersion with totally geodesic fibers. The K\"ahler structure on  $\mathbb{C}P^m$ is inherited from the one in $\mathbb{C}^{m+1}$ through this construction.

 To parametrize points in $\mathbb{C}P^m \setminus \{ \infty \}$, it is convenient to  use the local inhomogeneous coordinates given by $w_j=z_j/z_{m+1}$, $1 \le j \le m$, $z \in \mathbb{C}^{n+1}$, $z_{m+1}\neq 0$. The point $ \infty $ on $\mathbb{C}P^m$ corresponds to $z_{m+1}= 0$.

 The submersion $\pi$ allows one to construct the Brownian motion on $\mathbb{C}P^m$ from the Brownian motion on $\bS^{2m+1}$. Indeed,  let $(Z_t)_{t \ge 0}$ be a Brownian motion on the Riemannian sphere  $\bS^{2m+1} \subset \mathbb{C}^{m+1}$ started at the north pole \footnote{We  call north pole the point with complex coordinates $z_1=0,\cdots, z_{m+1}=1$. }. Since $\mathbb{P}( \exists t \ge 0, Z^{m+1}(t)=0 )=0$, one can use the local description of the submersion $\pi$ in inhomogeneous coordinates to deduce that
\begin{align}\label{BMsphere}
W_t= \left( \frac{Z^1_t}{Z^{m+1}_t} , \cdots, \frac{Z^m_t}{Z^{m+1}_t}\right), \quad t \ge 0,
\end{align}
is a Brownian motion on $\mathbb{C}P^m$, i.e. is a diffusion process with generator 
\[
\Delta_{\mathbb{C}P^m}=4(1+|w|^2)\sum_{k=1}^m \frac{\partial^2}{\partial w_k \partial\overline{w_k}}+ 4(1+|w|^2)\mathcal{R} \overline{\mathcal{R}}
\]
where
\[
\mathcal{R}=\sum_{j=1}^m w_j \frac{\partial}{\partial w_j}.
\]

 The radial part of $W$ defined by 
\[
r_t=\arctan| W_t|=\arctan \sqrt{ \sum_{i=1}^m \frac{| Z^i_t|^2}{| Z^{m+1}_t|^2}}=\arctan \left( \frac{1}{| Z^{m+1}_t|}\sqrt{1- | Z^{m+1}_t|^2 }\right)
\]
is  a diffusion process with Jacobi generator
\[
L_{\mathbb{C}P^m}=\frac{\partial^2}{\partial r^2}+((2m-2)\cot r +2 \cot 2 r)\frac{\partial}{\partial r}  .
\]
We note that $L_{\mathbb{C}P^m}=\mathcal{L}^{m-1,0}$ where $\mathcal{L}^{m-1,0}$ is the operator studied in the appendix of \cite{MR3719061}. In particular, the spectrum of $\mathbb{C}P^m$ is given by:
\[
\mathrm{Sp} ( \mathbb{C}P^m)=\left\{ 4k (k+m), k \ge 1 \right\}.
\]
Finally, we note that the radial part of the Riemannian volume measure writes
\[
d\mu_{\mathbb{C}P^m}= \frac{\pi^m}{(m-1)!} \, (\sin r)^{2m-2} \sin (2r) \,dr, \quad  0 \le r \le \frac{\pi}{2}. 
\]


\

\textbf{Negatively curved model.}  The negatively curved model of a K\"ahler manifold is the complex projective space $\mathbb CH^m$.  It can be constructed as follows. Let us consider the complex hyperboloid
\[
\mathcal H^{2m+1}=\{ z \in \mathbb{C}^{m+1}, | z_1|^2+\cdots+|z_m|^2 -|z_{m+1}|^2=-1 \} \subset \mathbb{C}^{m+1} .
\]

The group $\mathbf{U}(1)$ acts isometrically on $\mathcal H^{2m+1}$. The quotient space of $\mathcal H^{2m+1}$ by  this action is defined to be $\mathbb{C}H^m$ and the projection map $\pi :  \mathcal H^{2m+1} \to \mathbb CH^m$ is a Riemannian submersion with totally geodesic fibers. Thus, as a differential manifold, the  complex hyperbolic space $\mathbb{C}H^m$ is simply the open unit ball in $\mathbb{C}^m$ with a Riemannian metric inherited from the previous submersion. The K\"ahler structure on  $\mathbb{C}H^m$ is inherited from the one in $\mathbb{C}^{m+1}$ through the above construction.

To parametrize $\mathbb{C}H^m$, one can use the global inhomogeneous coordinates given by $w_j=z_j/z_{m+1}$ where $(z_1,\dots, z_{m+1})\in \mathcal H^{2m+1}$. In those coordinates the Laplace operator of $\mathbb{C}H^m$ can be written:

 \[
\Delta_{\mathbb{C}H^m}=4(1-|w|^2)\sum_{k=1}^m \frac{\partial^2}{\partial w_k \partial\overline{w_k}}+ 4(1-|w|^2)\mathcal{R} \overline{\mathcal{R}}
\]
where
\[
\mathcal{R}=\sum_{j=1}^m w_j \frac{\partial}{\partial w_j}.
\] 

The Brownian motion $(W_t)_{ t \ge 0}$ on $\mathbb{C}H^m$ is the diffusion with generator $\Delta_{\mathbb{C}H^m}$. As for the case of $\mathbb{C}P^m$, it may be represented in inhomogeneous coordinates as 
\[
W_t= \left( \frac{Z^1_t}{Z^{m+1}_t} , \cdots, \frac{Z^m_t}{Z^{m+1}_t}\right), \quad t \ge 0
\]
where $(Z^1_t,\cdots,Z^{m+1}_t)$ is a Brownian motion on $\mathcal H^{2m+1}$. The radial part of $W$ defined by 
\[
r_t=\arctanh | W_t|=\arctanh\sqrt{ \sum_{i=1}^m \frac{| Z^i_t|^2}{| Z^{m+1}_t|^2}}= \arctanh \left( \frac{1}{| Z^{m+1}_t|}\sqrt{| Z^{m+1}_t|^2 -1 }\right)
\]
is  a diffusion process with hyperbolic Jacobi generator

\[
L_{\mathbb{C}H^m}= \frac{\partial^2}{\partial r^2}+((2m-2)\coth r+2  \coth 2r)\frac{\partial}{\partial r}
 \]
 
 Finally, we note that the radial part of the Riemannian volume measure writes
\[
d\mu_{\mathbb{C}H^m}= \frac{\pi^m}{(m-1)!} \, (\sinh r)^{2m-2} \sinh (2r) \,dr, \quad  r \ge 0. 
\]

  
\subsubsection{Quaternion K\"ahler models}

\textbf{Flat model.} Let $\mathbb{H}$ be the non-commutative field of quaternions
\[
\mathbb{H}=\{q=t+xI+yJ+zK, (t,x,y,z)\in\R^4\},
\]
where  $I,J,K$  satisfy $I^2=J^2=K^2=IJK=-1$.  For $q=t+xI+yJ+zK \in \mathbb{H}$, we denote by $\overline q= t -xI-yJ-zK$ its conjugate, $|q|^2=t^2+x^2+y^2+z^2$ its squared norm and $\mathrm{Im}(q)=(x,y,z) \in \mathbb{R}^3$ its imaginary part.

The quaternionic structure $I,J,K$ in that case is  the component-wise multiplication by $I,J,K$ respectively. The Brownian motion $(W_t)_{t \ge 0}$ on  $\mathbb{H}^m$ is the diffusion process associated with the Laplace operator
\[
\Delta_{\mathbb{H}^m}=\sum_{i=1}^m \frac{\partial^2}{\partial t^2_i} +   \frac{\partial^2}{\partial x^2_i} +\frac{\partial^2}{\partial y^2_i} +\frac{\partial^2}{\partial z^2_i} 
\]
One can represent
\[
W_t =\left( Q^1_t,\cdots,Q^m \right),
\]
where the $Q^i$'s are independent complex Brownian motions on $\mathbb H$. The radial part of $W$ defined by 
\[
r_t=| W_t|=\sqrt{ \sum_{i=1}^m | Q^i_t|^2}
\]
is  a diffusion process with Bessel generator
\[
L_{\mathbb{H}^m}=\frac{\partial^2}{\partial r^2} + \frac{4m-1}{r} \frac{\partial}{\partial r} .
\]
We note that the radial part of the Lebesgue measure on $\mathbb{H}^m$ then writes:
\[
d\mu_{\mathbb{H}^m} = 2 \frac{\pi^{2m}}{(2m-1)!} \, r^{4m-1} dr, \quad r \ge 0.
\]

\

\textbf{Positively curved model.} The positively curved model of a K\"ahler manifold is the quaternionic projective space $\mathbb HP^m$. It can be constructed as follows. Consider the unit sphere
\[
\bS^{4m+3}=\lbrace q=(q_1,\cdots,q_{m+1})\in \mathbb{H}^{m+1}, \| q \| =1\rbrace.
\]
The group of unit quaternions is isomorphic to the Lie group $\mathbf{SU}(2)$.
Thus, there is an isometric group action of $\mathbf{SU}(2)$ on $\bS^{4m+3}$ which is  defined by $$q \cdot(q_1,\cdots, q_{m+1}) = (q q_1,\cdots, qq_{m+1}). $$

The quotient space $\bS^{4m+3} / \mathbf{SU}(2)$ is defined as the quaternionic projective space $\mathbb{H}P^m$ and the projection map $\pi :  \bS^{4m+3} \to \mathbb{H}P^m$ is a Riemannian submersion with totally geodesic fibers. The quaternion K\"ahler structure on  $\mathbb{H}P^m$ is inherited from the one in $\mathbb{H}^{m+1}$ through this construction.

 To parametrize points in $\mathbb{H}P^m \setminus \{ \infty \}$, we  use the local inhomogeneous coordinates given by $w_j=q_{m+1}^{-1} q_m$, $1 \le j \le m$, $q \in \mathbb{H}^{n+1}$, $q_{m+1}\neq 0$. The point $ \infty $ on $\mathbb{H}P^m$ corresponds to $q_{m+1}= 0$ and one can identify $\mathbb{H}P^m$ with $\mathbb{H}^m \cup \{ \infty \}$.

As before,  the submersion $\pi$ allows to construct the Brownian motion on $\mathbb{H}P^m$ from the Riemannian Brownian motion on $\bS^{4m+3}$. Indeed,  let $(Q_t)_{t \ge 0}$ be a Brownian motion on the Riemannian sphere  $\bS^{4m+3} \subset \mathbb{H}^{m+1}$ started at the north pole \footnote{We  call here north pole the point with quaternionic coordinates $q_1=0,\cdots, q_{m+1}=1$. }. Since $\mathbb{P}( \exists t \ge 0, Q^{m+1}(t)=0 )=0$, one deduces that
\begin{align}\label{BMquaternionsphere}
W_t= \left( (Q^{m+1}_t)^{-1} Q^1_t , \cdots, (Q^{m+1}_t)^{-1} Q^m_t \right), \quad t \ge 0,
\end{align}
is a Brownian motion on $\mathbb{H}P^m$, i.e. is a diffusion process with generator 

\[
\Delta_{\mathbb{H}P^m}=4(1+|w|^2)^2\sum_{k=1}^m \mathrm{Re} \left( \frac{\partial^2}{\partial w_k \partial\overline{w_k}} \right)- 8(1+|w|^2)\mathrm{Re} \left( \sum_{j=1}^m w_j \frac{\partial}{\partial w_j} \right)
\]

In real coordinates, we have $w_i=t_i+x_iI+y_iJ+z_iK$ and 
\begin{equation*}
\frac{\partial}{\partial w_i}:=\frac12\left(\frac{\partial}{\partial t_i}-\frac{\partial}{\partial x_i}I-\frac{\partial}{\partial y_i}J-\frac{\partial}{\partial z_i}K\right).
\end{equation*}

 The radial part of $W$ defined by 
\[
r_t=\arctan| W_t|=\arctan \left( \frac{1}{| Q^{m+1}_t|}\sqrt{1- | Q^{m+1}_t|^2 }\right)
\]
is  a diffusion process with Jacobi generator
\[
L_{\mathbb{H}P^m}=\frac{\partial^2}{\partial r^2}+((4m-4)\cot r+6\cot2 r)\frac{\partial}{\partial r} .
\]
We note that $L_{\mathbb{H}P^m}=\mathcal{L}^{2m-1,1}$ where $\mathcal{L}^{2m-1,1}$ is the operator studied in the appendix of \cite{MR3719061}. In particular, the spectrum of $\mathbb{H}P^m$ is given by:
\[
\mathrm{Sp} ( \mathbb{H}P^m)=\left\{ 4k (k+2m+1), k \ge 1 \right\}.
\]
Finally, we note that the radial part of the Riemannian volume measure writes
\[
d\mu_{\mathbb{H}P^m}= \frac{\pi^{2m}}{4(2m-1)!} \, (\sin r)^{4m-4} \sin (2r)^3 \,dr, \quad  0 \le r \le \frac{\pi}{2}. 
\]


\

\

\textbf{Negatively curved model.} The positively curved model of a K\"ahler manifold is the quaternionic hyperbolic space $\mathbb HH^m$. It can be constructed as follows. 
Let us consider the quaternionic hyperboloid
\[
\mathcal Q^{4m+3}=\{ q \in \mathbb{H}^{m+1}, | q_1|^2+\cdots+|q_m|^2 -|q_{m+1}|^2=-1 \} \subset \mathbb{H}^{m+1} .
\]

The group $\mathbf{SU}(2)$ acts isometrically on $\mathcal Q^{4m+3}$. The quotient space of $\mathcal Q^{4m+3}$ by  this action is defined to be $\mathbb{H}H^m$ and the projection map $\pi :  \mathcal Q^{4m+3} \to \mathbb HH^m$ is a Riemannian submersion with totally geodesic fibers. The quaternion K\"ahler structure on  $\mathbb{H}H^m$ is inherited from the one in $\mathbb{H}^{m+1}$.

To parametrize $\mathbb{H}H^m$, we use the global inhomogeneous coordinates given by $w_j=q^{-1}_{m+1}q_j$ where $(q_1,\dots, q_{m+1})\in \mathcal Q^{4m+3}$. In those coordinates the Laplace operator of $\mathbb{H}H^m$ can be written:

 \[
\Delta_{\mathbb{H}H^m}=4(1-|w|^2)^2 \sum_{k=1}^m \mathrm{Re} \left( \frac{\partial^2}{\partial w_k \partial\overline{w_k}}\right)+ 8(1+|w|^2)\mathrm{Re} \left( \sum_{j=1}^m w_j \frac{\partial}{\partial w_j} \right)
\]

The Brownian motion $(W_t)_{ t \ge 0}$ on $\mathbb{H}H^m$ is the diffusion with generator $\Delta_{\mathbb{H}H^m}$.  It can be represented as
\[
W_t= \left( (Q^{m+1}_t)^{-1} Q^1_t , \cdots, (Q^{m+1}_t)^{-1} Q^m_t \right), \quad t \ge 0
\]
where $(Q_t)_{t \ge 0}$ is a Brownian motion on $\mathcal Q^{4m+3}$.

The radial part of $W$ defined by 
\[
r_t=\arctanh | W_t|=\arctanh \left( \frac{1}{| Q^{m+1}_t|}\sqrt{| Q^{m+1}_t|^2-1 }\right)
\]
is  a diffusion process with hyperbolic Jacobi generator

\[
L_{\mathbb{H}H^m}= \frac{\partial^2}{\partial r^2}+((4m-4)\coth r+6  \coth 2r)\frac{\partial}{\partial r}
 \]
Finally, we note that the radial part of the Riemannian volume measure writes
\[
d\mu_{\mathbb{H}H^m}= \frac{\pi^{2m}}{4(2m-1)!} \, (\sinh r)^{4m-4} \sinh (2r)^3 \,dr, \quad  r \ge 0. 
\]

We refer to \cite{BDW1,BDW2} and references therein for complementary details.

\subsection{Summary of the model spaces}\label{summary}

For later use, and as a summary, we collect the results about the model spaces that will be used later. Additionally, in those model spaces the holomorphic/quaternionic sectional curvatures and orthogonal Ricci curvatures defined earlier may be computed explicitly and yield the following results:

%
%

\begin{table}[H]
\centering
\scalebox{0.8}{
\begin{tabular}{|p{1.5cm}||p{8cm}|c|>{\centering\arraybackslash}p{1.8cm}|c|  }
  \hline
 $\M$ &   Radial Laplacian   &   Radial measure      \\
  \hline
 \hline
$\mathbb{C}^m$ & $ L_{\mathbb{C}^m}=\frac{\partial^2}{\partial r^2} + \frac{2m-1}{r} \frac{\partial}{\partial r} $ & $d\mu_{\mathbb{C}^m} = 2 \frac{\pi^m}{(m-1)!} \, r^{2m-1} dr$  \\
 $\mathbb{C}P^m$ &$L_{\mathbb{C}P^m}=\frac{\partial^2}{\partial r^2}+((2m-2)\cot r +2 \cot 2 r)\frac{\partial}{\partial r} $ & $d\mu_{\mathbb{C}P^m}= \frac{\pi^m}{(m-1)!} \, (\sin r)^{2m-2} \sin (2r) \,dr$    \\
$\mathbb{C}H^m$  & $L_{\mathbb{C}H^m}= \frac{\partial^2}{\partial r^2}+((2m-2)\coth r+2  \coth 2r)\frac{\partial}{\partial r}$ & $d\mu_{\mathbb{C}H^m}= \frac{\pi^m}{(m-1)!} \, (\sinh r)^{2m-2} \sinh (2r) \,dr$   \\
  \hline
\end{tabular}}
\caption{Radial Laplacians in K\"ahler model spaces.}
\label{Table 1}
\end{table}

\begin{table}[H]
\centering
\scalebox{0.8}{
\begin{tabular}{|p{1.5cm}||p{1.5cm}|c|>{\centering\arraybackslash}p{1.8cm}|c|  }
  \hline
 $\M$ &  $H$ & $\mathrm{Ric}^\perp$  \\
  \hline
 \hline
$\mathbb{C}^m$ &  $0 $ & $0$ \\
 $\mathbb{C}P^m$ & 4 & $2m-2$  \\
$\mathbb{C}H^m$  &  -4 & $-(2m-2)$  \\
  \hline
\end{tabular}}
\caption{Curvatures of K\"ahler model spaces.}
\label{Table 2}
\end{table}

 \begin{table}[H]
\centering
\scalebox{0.8}{
\begin{tabular}{|p{1.5cm}||p{8cm}|c|>{\centering\arraybackslash}p{1.8cm}|c|  }
  \hline
 $\M$ &   Radial Laplacian   &   Radial measure      \\
  \hline
 \hline
$\mathbb{H}^m$ & $ L_{\mathbb{H}^m}=\frac{\partial^2}{\partial r^2} + \frac{4m-1}{r} \frac{\partial}{\partial r} $ & $d\mu_{\mathbb{H}^m} = 2 \frac{\pi^{2m}}{(2m-1)!} \, r^{4m-1} dr$ \\
 $\mathbb{H}P^m$ &$L_{\mathbb{H}P^m}=\frac{\partial^2}{\partial r^2}+((4m-4)\cot r+6\cot2 r)\frac{\partial}{\partial r}$ & $d\mu_{\mathbb{H}P^m}= \frac{\pi^{2m}}{4(2m-1)!} \, (\sin r)^{4m-4} \sin (2r)^3 \,dr$    \\
$\mathbb{H}H^m$  & $L_{\mathbb{H}H^m}= \frac{\partial^2}{\partial r^2}+((4m-4)\coth r+6  \coth 2r)\frac{\partial}{\partial r}$ & $d\mu_{\mathbb{H}H^m}= \frac{\pi^{2m}}{4(2m-1)!} \, (\sinh r)^{4m-4} \sinh (2r)^3 \,dr$   \\
  \hline
\end{tabular}}
\caption{Radial Laplacians in quaternion K\"ahler model spaces.}
\label{Table 3}
\end{table}

 \begin{table}[H]
\centering
\scalebox{0.8}{
\begin{tabular}{|p{1.5cm}||p{1.5cm}|c|>{\centering\arraybackslash}p{1.8cm}|c|  }
  \hline
 $\M$ &   $Q$ & $\mathrm{Ric}^\perp$  \\
  \hline
 \hline
$\mathbb{H}^m$ &  $0 $ & $0$ \\
 $\mathbb{H}P^m$ & 12 & $4m-4$  \\
$\mathbb{H}H^m$  &  -12 & $-(4m-4)$  \\
  \hline
\end{tabular}}
\caption{Curvatures of the quaternion K\"ahler model spaces.}
\label{Table 4}
\end{table}

\section{Laplacian comparison theorems}

This subsection is devoted to the proofs of the sharp Laplace comparison theorems in K\"ahler and quaternion K\"ahler manifolds. The main technical tool is the classical index lemma. In the K\"ahler case, the comparison theorem is due to Ni-Zheng \cite{MR3858834} but seems to be new in the quaternion K\"ahler case.

We introduce the comparison function.

\begin{equation}
F(k,r) = \begin{cases}  \sqrt{k} \cot \sqrt{k} r & \text{if $k > 0$,} \\
\frac{1}{r} & \text{if $k = 0$,}\\ \sqrt{|k|} \coth \sqrt{|k|} r & \text{if $k < 0$.} \end{cases}
\end{equation}

\subsection{K\"ahler case}

Let $(\M,g,J)$ be a complete K\"ahler with complex dimension $m$ (i.e. the real dimension is $2m$). We denote by $d(x,y)$ the Riemannian distance between $x,y \in \M$ and by $\Delta$ the Laplace-Beltrami operator on $\M$. The following Laplacian comparison theorem was proved in \cite{MR3858834}. As before, we denote by $H$ the holomorphic sectional curvature of $\M$ and by $\mathrm{Ric}^\perp$ its orthogonal Ricci curvature.

\begin{theorem}[Ni-Zheng \cite{MR3858834}]\label{comparison kahler}
Let $k \in \mathbb{R}$. Assume that $H \ge 4k$ and that $\mathrm{Ric}^\perp \ge (2m-2)k$. Let $x_0 \in \M$ and denote $r(x)=d(x_0,x)$. Then, pointwise outside of the cut-locus of $x_0$, and everywhere in the sense of distributions, one has
\[
\Delta r \le (2m-2) F(k,r) + 2F(k,2r).
\]
\end{theorem}

\begin{proof}
The result can be found in \cite{MR3858834}. We provide here a self-contained proof not only for completeness but also because the structure of our proof will be generalized to the quaternionic K\"ahler case which is new.

We can assume $m \ge 2$ since for the case $m=1$, the statement reduces to the classical Laplacian comparison theorem in Riemannian geometry. Let $x_0 \in \M$ and $x \neq x_0$ which is not in the cut-locus of $x$.  Let $\gamma:[0,r(x)] \to \M$ be the unique length parametrized geodesic connecting $x_0$ to $x$ . At $x$, we consider an orthonormal frame $\{ X_1(x),\cdots, X_{2m}(x) \}$ such that
\[
X_1(x) =\gamma'(r(x)), \, X_2(x)=J \gamma'(r(x)).
\]

We have then
\[
\Delta r (x) =\sum_{i=1}^{2m} \nabla^2 r ( X_i(x),X_i(x)).
\]
We divide the above sum into three parts: $\nabla^2 r ( X_1(x),X_1(x))$, $\nabla^2 r ( X_2(x),X_2(x))$ and $\sum_{i=3}^{2m} \nabla^2 r ( X_i(x),X_i(x))$.  The first term $\nabla^2 r ( X_1(x),X_1(x))$ is zero because $X_1(x) =\gamma'(r(x))$.
We now estimate the second term. Note that the vector field defined along $\gamma$ by $J \gamma'$ is parallel because $J$ is parallel and $\gamma$ is a geodesic, thus satisfies $\nabla_{\gamma'} \gamma'=0$. We consider then the vector field defined along $\gamma$ by
\[
\tilde{X} (\gamma(t)) =\frac{\mathfrak{s}(4k,t)}{\mathfrak{s}(4k,r(x))} J \gamma'(t).
\]
where
\begin{equation}
\mathfrak{s}(k,t) = \begin{cases}   \sin  \sqrt{k} t & \text{if $k > 0$,} \\
t & \text{if $k = 0$,}\\  \sinh \sqrt{|k|} t & \text{if $k < 0$.} \end{cases}
\end{equation}

From the index lemma  we have
\begin{align*}
\nabla^2 r ( X_2(x),X_2(x))  & \le  \int_0^{r(x)} \left(\langle \nabla_{\gamma'}  \tilde X, \nabla_{\gamma'} \tilde X  \rangle-\langle R(\gamma',\tilde X) \tilde X,\gamma'  \rangle \right) dt \\
 & \le \frac{1}{\mathfrak{s}(4k,r(x))^2} \int_0^{r(x)}\left( \mathfrak{s}'(4k,t)^2-\mathfrak{s}(4k,t)^2 \langle R(\gamma',J\gamma') J \gamma',\gamma'  \rangle \right) dt  \\
 & \le \frac{1}{\mathfrak{s}(4k,r(x))^2} \int_0^{r(x)}\left( \mathfrak{s}'(4k,t)^2-4k\mathfrak{s}(4k,t)^2 \right) dt \\
 &\le 2F(k,2r(x)).
\end{align*}

Finally, we estimate the last term $\sum_{i=3}^{2m} \nabla^2 r ( X_i(x),X_i(x))$. 
In order to proceed, we denote by $\{ X_3,\cdots, X_{2m} \}$ the vector fields along $\gamma$ obtained by parallel transport of $\{ X_3(x),\cdots, X_{2m}(x) \}$. We observe that everywhere along $\gamma$, the family
\[
\{ \gamma', J \gamma', X_3,\cdots, X_{2m} \}
\]
is an orthonormal frame. We consider then the vector field defined along $\gamma$ by
\[
\tilde{X}_i  (\gamma(t)) =\frac{\mathfrak{s}(k,t)}{\mathfrak{s}(k,r(x))} X_i (\gamma(t)), \, i=3, \cdots, 2m.
\]
From the index lemma we obtain
\begin{align*}
\sum_{i=3}^{2m} \nabla^2 r ( X_i(x),X_i(x)) & \le \sum_{i=3}^{2m}  \int_0^{r(x)} \left(\langle \nabla_{\gamma'}  \tilde X_i, \nabla_{\gamma'} \tilde X_i  \rangle-\langle R(\gamma',\tilde X_i) \tilde X_i,\gamma'  \rangle \right) dt \\
 & \le \frac{1}{\mathfrak{s}(k,r(x))^2} \sum_{i=3}^{2m} \int_0^{r(x)}\left( \mathfrak{s}'(k,t)^2-\mathfrak{s}(k,t)^2 \langle R(\gamma',\tilde X_i') \tilde X_i,\gamma'  \rangle \right) dt  \\
 & \le \frac{1}{\mathfrak{s}(k,r(x))^2}  \int_0^{r(x)}\left( (2m-2) \mathfrak{s}'(k,t)^2-\mathfrak{s}(k,t)^2 \sum_{i=3}^{2m}  \langle R(\gamma',\tilde X_i') \tilde X_i,\gamma'  \rangle \right) dt  \\
 & \le \frac{1}{\mathfrak{s}(k,r(x))^2}  \int_0^{r(x)}\left( (2m-2) \mathfrak{s}'(k,t)^2-\mathfrak{s}(k,t)^2 \mathrm{Ric}^\perp (\gamma',\gamma') \right) dt  \\
 & \le \frac{2m-2}{\mathfrak{s}(k,r(x))^2} \int_0^{r(x)}\left( \mathfrak{s}'(k,t)^2-k\mathfrak{s}(k,t)^2 \right) dt \\
 &\le (2m-2)F(k,r(x)).
\end{align*}
Therefore we conlude
\[
\Delta r (x)\le (2m-2) F(k,r(x)) + 2F(k,2r(x)).
\]
Finally, proving that  everywhere in the sense of distributions, one has
\[
\Delta r \le (2m-2) F(k,r) + 2F(k,2r).
\]
is similar to the corresponding proof in the Riemannian case (which relies on Calabi lemma), so we skip the details.
\end{proof}

It is remarkable that the theorem is sharp on the model spaces $\mathbb{C}^m, \mathbb{C}P^m$ and $\mathbb{C}H^m$. On $\mathbb{C}^m$, one has $k=0$ and
\[
(2m-2) F(k,r) + 2F(k,2r)= \frac{2m-1}{r}.
\]
On $\mathbb{C}P^m$, one has $k=1$ and
\[
(2m-2) F(k,r) + 2F(k,2r)= (2m-2)\cot r +2 \cot 2 r,
\]
and on $\mathbb{C}H^m$, one has $k=-1$ and
\[
(2m-2) F(k,r) + 2F(k,2r)=(2m-2)\coth r +2 \coth 2 r.
\]

\subsection{Quaternion K\"ahler case}

Let now $(\M,g,I,J,K)$ be a complete quaternion K\"ahler with quaternionic dimension $m$ (i.e. the real dimension is $4m$). We also denote by $d(x,y)$ the Riemannian distance between $x,y \in \M$ and by $\Delta$ the Laplace-Beltrami operator on $\M$. As before, we denote by $Q$ the quaternionic sectional curvature of $\M$ and by $\mathrm{Ric}^\perp$ its orthogonal Ricci curvature.

\begin{theorem}\label{comparison quaternion}
Let $k \in \mathbb{R}$. Assume that $Q \ge 12k$ and that $\mathrm{Ric}^\perp \ge (4m-4)k$. Let $x_0 \in \M$ and denote $r(x)=d(x_0,x)$. Then, pointwise outside of the cut-locus of $x_0$, and everywhere in the sense of distributions, one has
\[
\Delta r \le (4m-4) F(k,r) + 6F(k,2r).
\]
\end{theorem}

\begin{proof}
The proof proceeds as in the K\"ahler case but is slightly more involved.
As before, we can assume $m \ge 2$ since for the case $m=1$, the statement reduces to the classical Laplacian comparison theorem in Riemannian geometry. Let $x_0 \in \M$ and $x \neq x_0$ which is not in the cut-locus of $x$.  Let $\gamma:[0,r(x)] \to \M$ be the unique length parametrized geodesic connecting $x_0$ to $x$ 
At $x$, we consider an orthonormal frame $\{ X_1(x),\cdots, X_{4m}(x) \}$ such that
\[
X_1(x) =\gamma'(r(x)), \, X_2(x)=I \gamma'(r(x)), \, X_3(x)=J \gamma'(r(x)), \,  X_4(x)=K \gamma'(r(x))
\]

We have then
\[
\Delta r (x) =\sum_{i=1}^{4m} \nabla^2 r ( X_i(x),X_i(x)).
\]

We divide the above sum into three parts: $\nabla^2 r ( X_1(x),X_1(x))$, $\sum_{i=2}^4 \nabla^2 r ( X_i(x),X_i(x))$ and $\sum_{i=5}^{4m} \nabla^2 r ( X_i(x),X_i(x))$.  The first term $\nabla^2 r ( X_1(x),X_1(x))$ is zero because $X_1(x) =\gamma'(r(x))$.
Estimating the second term requires more work than in the K\"ahler case, because the  vectors  $I \gamma', J \gamma'$ and $K \gamma'$ might not be parallel along $\gamma$. Let us denote by $X_2,X_3$ and $X_4$ the vector fields along $\gamma$ obtained by parallel transport along $\gamma$ of $X_2 (x),X_3(x)$ and $X_4(x)$. Since along $\gamma$ one has
\[
\nabla_{\gamma'} I, \nabla_{\gamma'} J, \nabla_{\gamma'} K \in \mathbf{span} \{ I,J,K\} 
\]
we deduce that along $\gamma$ one has
\[
\mathbf{span} \{ X_2, X_3, X_4 \}= \mathbf{span} \{ I \gamma' ,J \gamma' , K \gamma' \}.
\]
Moreover $\{ X_2, X_3, X_4 \}$ and $\{ I \gamma' ,J \gamma' , K \gamma' \}$ are both orthonormal along $\gamma$. One deduces
\begin{align*}
 & R(\gamma',X_2,X_2,\gamma')+R(\gamma',X_3,X_3,\gamma')+R(\gamma',X_4,X_4,\gamma') \\
=& R(\gamma',I\gamma',I\gamma',\gamma')+R(\gamma',J\gamma',J\gamma',\gamma')+R(\gamma',K\gamma',K\gamma',\gamma') \\
=& Q (\gamma').
\end{align*}
As a consequence, if we consider  the vector field defined along $\gamma$ by
\[
\tilde{X}_i  (\gamma(t)) =\frac{\mathfrak{s}(4k,t)}{\mathfrak{s}(4k,r(x))} X_i (\gamma(t)), \, i=2,3,4 ,
\]
we obtain by the same computation as in the proof of theorem \ref{comparison kahler}
\[
\sum_{i=2}^4 \nabla^2 r ( X_i(x),X_i(x)) \le 6F(k,2r(x)).
\]
The estimate of the term  $\sum_{i=5}^{4m} \nabla^2 r ( X_i(x),X_i(x))$ is similar as in the proof of theorem \ref{comparison kahler}, so we skip the details for conciseness.
\end{proof}

As in the K\"ahler case, it is remarkable that the theorem is sharp on the model spaces $\mathbb{H}^m, \mathbb{H}P^m$ and $\mathbb{H}H^m$.

\section{ Comparison theorems for radial processes and applications}

\subsection{It\^o formula for radial processes on Riemannian manifolds}

To fix notations, we first recall the well-known Kendall theorem \cite{kendall} about the It\^o formula for the radial parts of Brownian motions on a Riemannian manifold. Throughout this subsection $(\M,g)$ is a complete Riemannian manifold and $\Delta$ denotes the Laplace-Beltrami operator. Let $( ( X_t )_{t \ge 0} , ( \mathbb P_x )_{x \in \M} )$ be the diffusion process generated by $\Delta$, i.e. the Brownian motion on $\M$.  Take $x_0 \in \M$ and set $r (x) := d ( x_0 , x )$. 
We denote by $\Cut (x_0)$ the cut-locus of $x_0$.  Let $\zeta$ be the life time of $X$. 

\begin{theorem}[Kendall  \cite{kendall}] \label{th:Ito-radial}
For each $x_1 \in \M$, 
there exist a non-decreasing continuous process $l_t$ 
which increases only when $X_t \in \Cut (x_0)$ 
and a Brownian motion $\be_t$ on $\R$ with $\ang{\be}_t = 2 t$
such that 
\begin{equation} \label{eq:Ito-radial} 
r ( X_{ t \wedge \zeta } ) 
= 
r ( X_0 ) 
+ \be_t 
+ \int_0^{ t \wedge \zeta } \Delta r ( X_s ) ds
- l_{ t \wedge \zeta }
\end{equation}
holds $\mathbb P_{x_1}$-almost surely. 
\end{theorem}

\subsection{Comparison theorems on K\"ahler manifolds}

Let $(\M,g,J)$ be a complete K\"ahler with complex dimension $m$  Let $( ( X_t )_{t \ge 0} , ( \mathbb{P}_x )_{x \in \M} )$ be the  Brownian motion on $\M$.  As before, we fix a point $x_0 \in \M$.  For $x_1 \in \M$, we consider
the solution of the stochastic differential equation
\[
  \rho^k_t=d (x_0,x_1)+\int_0^t \left( (2m-2) F(k,\rho^k_s) + 2F(k,2\rho^k_s)\right) 
   ds+\sqrt{2} \beta_t
\]
where $\beta$ is a standard Brownian motion under $\mathbb{P}_{x_1}$.

With Laplacian comparison theorems and It\^o's formula \eqref{eq:Ito-radial} in hands, it is possible to apply \textit{mutatis mutandis} the  general available comparison methods  developed in the Riemannian case for instance by Ichihara \cite{ichihara}. We also refer to sections 3.5, 3.6 and 4.5 in the book \cite{MR1882015} by Hsu. This yields the following basic comparison result.

\begin{theorem}\label{comparison kahler  1d}
 Let $k \in \mathbb{R}$. Assume that $H \ge 4k$ and that $\mathrm{Ric}^\perp \ge (2m-2)k$. Then, for $x_1 \in \M$, $R>0$, and $s \le R$
\[
  \mathbb{P}_{x_1}\big\{ d (x_0,X_t) <s , \, t \le \tau_R \big\} \ge \mathbb{P}_{x_1}\big\{ \rho^k_t<s, \, t \le \tau^k_R \big\},
\]
where $\tau_R$ is the hitting time of the  geodesic ball in $\M$ with center $x_0$ and radius $R$ and $\tau^k_R$ the hitting time of the level $R$ by $\rho^k$.
\end{theorem}

\subsubsection{Cheeger-Yau type lower bound for the heat kernel}

A first corollary of theorem \ref{comparison kahler  1d} is a Cheeger-Yau type lower bound for the heat kernel.   It gives a sharp lower bound for the Dirichlet heat kernel on balls  in terms of the heat kernel of a corresponding K\"ahler model space.

We introduce the following notation. For $k \in \mathbb{R}$, let $L_k$ be the diffusion operator given by
\begin{align*}
L_k=
\begin{cases}
\frac{\partial^2}{\partial r^2}+((2m-2)\sqrt{k} \cot \sqrt{k} r +2\sqrt{k} \cot 2 \sqrt{k}r)\frac{\partial}{\partial r}  & \text{ if } k >0 \\
\frac{\partial^2}{\partial r^2}+\frac{2m-1}{r} \frac{\partial}{\partial r}    &   \text{ if } k =0  \\
\frac{\partial^2}{\partial r^2}+((2m-2) \sqrt{|k|}\coth \sqrt{|k|} r +2 \sqrt{|k|} \coth 2  \sqrt{|k|} r)\frac{\partial}{\partial r}   & \text{ if } k <0
\end{cases}
\end{align*}

 and let $\mu_k$ be the measure
 
 \begin{align*}
d\mu_k=
\begin{cases}
\frac{\pi^m}{(m-1)! k^{m-1/2} } \, (\sin \sqrt{k}  r)^{2m-2} \sin (2\sqrt{k} r) \,dr & \text{ if } k >0 \\
2 \frac{\pi^m}{(m-1)!} \, r^{2m-1} dr    &   \text{ if } k =0  \\
\frac{\pi^m}{(m-1)! |k|^{m-1/2} } \, (\sinh \sqrt{|k|}  r)^{2m-2} \sinh (2\sqrt{|k|} r) \,dr  & \text{ if } k <0.
\end{cases}
\end{align*}

Note that the operator $L_k$ is symmetric with respect to the measure $\mu_k$.
With the notations of section \ref{summary}, we  have
\[
(L_{-1}, \mu_{-1}) = ( L_{\mathbb C H^m} , \mu_{\mathbb C H^m}),  \quad (L_{0}, \mu_{0}) = ( L_{\mathbb C^m} , \mu_{\mathbb C^m}), \quad (L_{1}, \mu_{1}) = ( L_{\mathbb C P^m} , \mu_{\mathbb C P^m}).
\]
Moreover, depending on the sign of $k$, $(L_k,\mu_k)$ is obtained from $(L_1,\mu_1)$, $(L_{0}, \mu_{0})$ or $(L_{-1}, \mu_{1})$ by a simple rescaling by $\sqrt{|k|}$. 
 
\begin{theorem}[Cheeger-Yau type heat kernel lower bound]\label{comparison heat kernel}
Let $k \in \mathbb{R}$. Assume that $H \ge 4k$ and that $\mathrm{Ric}^\perp \ge (2m-2)k$. Let $R>0$. Let
$( ( X^R_t )_{t \ge 0} , ( \mathbb P_x )_{x \in B(x_0,R)} )$ be a  Brownian motion on $B(x_0,R)$  with Dirichlet boundary condition.  Let $p^R(t,x,y)$ be its heat kernel with respect to the Riemannian volume measure $\mu$. Let now
$q_k^R(t,r_1,r_2)$ be the heat kernel with respect to $\mu_k$  of the diffusion on $[0,R]$ with generator $L_k$ and Dirichlet boundary condition at $R$. Then, for every 
$t>0$ and $x_1\in B (x_0,R)$
\[
   p^R(t,x_0,x_1)  \ge  q_k^R(t,0,d(x_0,x_1)).
\]
\end{theorem}

\begin{proof}
From theorem \ref{comparison kahler  1d}, one has
\[
  \int_{B(x_0,s)} p^R(t,x_1,y) d\mu(y) \ge \int_0^s
  q_k^R(t,d(x_0,x_1),r)d\mu_k (r).
\]
When $ s\to 0^+$, one has
\[
\mu (B(x_0,s)) \sim \frac{\pi^m}{m!} s^{2m} \sim \mu_k ([0,s]).
\]
On the other hand, from the Lebesgue differentiation theorem one has
\[
\lim_{s \to 0^+}\frac{1}{\mu (B(x_0,s))}  \int_{B(x_0,s)} p^R(t,x_1,y) d\mu(y)= p^R(t,x_1,x_0)=p^R(t,x_0,x_1)
\]
and
\[
\lim_{s \to 0^+} \frac{1}{\mu_k ([0,s])} \int_0^s
  q_k^R(t,d(x_0,x_1),r)d\mu_k (r)=q_k^R(t,d(x_0,x_1),0)=q_k^R(t,0,d(x_0,x_1)).
\]
The conclusion follows.
\end{proof}

\subsubsection{Cheng's estimates for Dirichlet eigenvalues on metric balls}

A nice corollary of the Cheeger-Yau's type heat kernel lower bound is a Cheng's type upper bound for the Dirichlet eigenvalues of Riemannian balls in terms of the eigenvalues of Riemannian balls in the corresponding K\"ahler model.

\begin{proposition}[Cheng's type estimates]
Let $k \in \mathbb{R}$. Assume that $H \ge 4k$ and that $\mathrm{Ric}^\perp \ge (2m-2)k$. Let $R>0$. For $x_0 \in \M$ let $\lambda_1( B_0(x_0,R))$ denote the
  first Dirichlet eigenvalue of the Riemannian ball $B(x_0,R)$ and let $\lambda_1(m,k,R)$ denote the first Dirichlet
  eigenvalue of the operator $L_k$  on the interval $[0,R]$ with Dirichlet boundary condition at
  $R$. Then, for every $x_0 \in \M$ and $R>0$
  \[
    \lambda_1( B(x_0,R)) \le \lambda_1(m, k,R).
  \]
\end{proposition}

\begin{proof}
  From spectral theory, one has
  \[
    p^R(t,x_1,y) =\sum_{j=1}^{+\infty} e^{-\lambda _j t} \phi_j (x_1)
    \phi_j(y)
  \]
  where the $\lambda_j$'s are the Dirichlet eigenvalues of
  $B(x_0,R)$ and the $\phi_j$'s the eigenfunctions.  One has a similar spectral expansion for $q_k^R(t,r_0,r)$.  Thus, from Corollary \ref{comparison heat kernel}, when
  $t \to +\infty$ one must have $\lambda_1 \le \tilde{\lambda}_1$.
\end{proof}


\subsection{Comparison theorems on quaternion K\"ahler manifolds}

In the quaternionic K\"ahler framework the comparison theorems of Cheeger-Yau's type and of Cheng's type might be obtained in a similar way as in the K\"ahler case.  The difference is the model diffusion with respect to which the comparison is made. 

Let  $(\M,g,I,J,K)$ be a complete quaternion K\"ahler with quaternionic dimension $m$ and for $k \in \mathbb{R}$ consider the following diffusion operator

\begin{align*}
\tilde L_k=
\begin{cases}
\frac{\partial^2}{\partial r^2}+((4m-4)\sqrt{k} \cot \sqrt{k} r +6\sqrt{k} \cot 2 \sqrt{k}r)\frac{\partial}{\partial r}  & \text{ if } k >0 \\
\frac{\partial^2}{\partial r^2}+\frac{4m-1}{r} \frac{\partial}{\partial r}    &   \text{ if } k =0  \\
\frac{\partial^2}{\partial r^2}+((4m-4) \sqrt{|k|}\coth \sqrt{|k|} r +6 \sqrt{|k|} \coth 2  \sqrt{|k|} r)\frac{\partial}{\partial r}   & \text{ if } k <0
\end{cases}
\end{align*}

 and measure
 
 \begin{align*}
d\tilde \mu_k=
\begin{cases}
 \frac{\pi^{2m}}{4(2m-1)! k^{2m-1/2}} \, (\sin \sqrt{k} r)^{4m-4} \sin (2\sqrt{k}r)^3 \,dr& \text{ if } k >0 \\
2 \frac{\pi^{2m}}{(2m-1)!} \, r^{4m-1} dr    &   \text{ if } k =0  \\
\frac{\pi^m}{(m-1)! |k|^{2m-1/2} } \, (\sinh \sqrt{|k|}  r)^{4m-4} \sinh (2\sqrt{|k|} r)^3 \,dr  & \text{ if } k <0.
\end{cases}
\end{align*}

Note that the operator $\tilde L_k$ is symmetric with respect to the measure $\tilde \mu_k$ and that with the notations of section \ref{summary}, we therefore have
\[
(\tilde L_{-1}, \tilde \mu_{-1}) = ( L_{\mathbb H H^m} , \mu_{\mathbb H H^m}),  \quad (\tilde L_{0}, \tilde \mu_{0}) = ( L_{\mathbb H^m} , \mu_{\mathbb H^m}), \quad (\tilde L_{1}, \tilde \mu_{1}) = ( L_{\mathbb H P^m} , \mu_{\mathbb H P^m}).
\]
As in the K\"ahler case, depending on the sign of $k$, $(\tilde L_k,\tilde \mu_k)$ is obtained from $(\tilde L_1,\tilde \mu_1)$, $(\tilde L_{0}, \tilde \mu_{0})$ or $(\tilde L_{-1}, \tilde \mu_{1})$ by a simple rescaling by $\sqrt{|k|}$

By applying the same methods as before, we obtain the following results.

\begin{theorem}[Cheeger-Yau type lower bound]\label{comparison heat kernel 2}
Let $k \in \mathbb{R}$.  Assume that $Q \ge 12k$ and that $\mathrm{Ric}^\perp \ge (4m-4)k$. Let $R>0$.  Let
$( ( X^R_t )_{t \ge 0} , ( \mathbb P_x )_{x \in B(x_0,R)} )$ be a  Brownian motion on $B(x_0,R)$  with Dirichlet boundary condition.  Let $p^R(t,x,y)$ be its heat kernel with respect to the Riemannian volume measure $\mu$. Let now
$\tilde q_k^R(t,r_1,r_2)$ be the heat kernel with respect to $\tilde \mu_k$  of the diffusion on $[0,R]$ with generator $\tilde L_k$ and Dirichlet boundary condition at $R$. Then, for every 
$t>0$ and $x_1\in B (x_0,R)$
\[
   p^R(t,x_0,x_1)  \ge  \tilde q_k^R(t,0,d(x_0,x_1)).
\]
\end{theorem}

\begin{proposition}[Cheng's type estimates]
Let $k \in \mathbb{R}$.  Assume that $Q \ge 12k$ and that $\mathrm{Ric}^\perp \ge (4m-4)k$. Let $R>0$. For $x_0 \in \M$ let $\lambda_1( B_0(x_0,R))$ denote the
  first Dirichlet eigenvalue of the Riemannian ball $B(x_0,R)$ and let $\tilde{\lambda}_1(m,k,R)$ denote the first Dirichlet
  eigenvalue of the operator $\tilde L_k$  on the interval $[0,R]$ with Dirichlet boundary condition at
  $R$. Then, for every $x_0 \in \M$ and $R>0$
  \[
    \lambda_1( B(x_0,R)) \le \tilde{\lambda}_1(m, k,R).
  \]
\end{proposition}

\bibliographystyle{amsplain}
\bibliography{references}

\providecommand{\bysame}{\leavevmode\hbox to3em{\hrulefill}\thinspace}
\providecommand{\MR}{\relax\ifhmode\unskip\space\fi MR }
\providecommand{\MRhref}[2]{%
  \href{http://www.ams.org/mathscinet-getitem?mr=#1}{#2}
}
\providecommand{\href}[2]{#2}
\begin{thebibliography}{10}

\bibitem{B02}
Fabrice Baudoin, \emph{Skew-product decompositions of {B}rownian motions on
  manifolds: a probabilistic aspect of the {L}ichnerowicz-{S}zabo theorem},
  Bull. Sci. Math. \textbf{126} (2002), no.~6, 481--491. \MR{1931625}

\bibitem{BDW2}
Fabrice Baudoin, Nizar Demni, and Jing Wang, \emph{Quaternionic brownian
  windings}, 2019.

\bibitem{BDW1}
\bysame, \emph{Quaternionic stochastic areas}, 2019.

\bibitem{baudoin2020radial}
Fabrice Baudoin, Erlend Grong, Kazumasa Kuwada, Robert Neel, and Anton
  Thalmaier, \emph{Radial processes for sub-riemannian brownian motions and
  applications}, 2020.

\bibitem{MR3719061}
Fabrice Baudoin and Jing Wang, \emph{Stochastic areas, winding numbers and
  {H}opf fibrations}, Probab. Theory Related Fields \textbf{169} (2017),
  no.~3-4, 977--1005. \MR{3719061}

\bibitem{besse}
Arthur~L. Besse, \emph{Einstein manifolds}, Classics in Mathematics,
  Springer-Verlag, Berlin, 2008, Reprint of the 1987 edition. \MR{2371700}

\bibitem{MR615626}
Jeff Cheeger and Shing~Tung Yau, \emph{A lower bound for the heat kernel},
  Comm. Pure Appl. Math. \textbf{34} (1981), no.~4, 465--480. \MR{615626}

\bibitem{MR378001}
Shiu~Yuen Cheng, \emph{Eigenvalue comparison theorems and its geometric
  applications}, Math. Z. \textbf{143} (1975), no.~3, 289--297. \MR{378001}

\bibitem{MR1882015}
Elton~P. Hsu, \emph{Stochastic analysis on manifolds}, Graduate Studies in
  Mathematics, vol.~38, American Mathematical Society, Providence, RI, 2002.
  \MR{1882015}
%

\bibitem{ichihara}
Kanji Ichihara, \emph{Comparison theorems for {B}rownian motions on
  {R}iemannian manifolds and their applications}, J. Multivariate Anal.
  \textbf{24} (1988), no.~2, 177--188. \MR{926351}

\bibitem{kendall}
Wilfrid~S. Kendall, \emph{Nonnegative {R}icci curvature and the {B}rownian
  coupling property}, Stochastics \textbf{19} (1986), no.~1-2, 111--129.
  \MR{864339}

\bibitem{MR3950013}
Kazumasa Kuwada and Kazuhrio Kuwae, \emph{Radial processes on RCD(K,N) spaces},
  J. Math. Pures Appl. (9) \textbf{126} (2019), 72--108. \MR{3950013}

\bibitem{MR3858834}
Lei Ni and Fangyang Zheng, \emph{Comparison and vanishing theorems for
  {K}\"{a}hler manifolds}, Calc. Var. Partial Differential Equations
  \textbf{57} (2018), no.~6, Art. 151, 31. \MR{3858834}

\bibitem{MR3996490}
\bysame, \emph{On orthogonal {R}icci curvature}, Advances in complex geometry,
  Contemp. Math., vol. 735, Amer. Math. Soc., Providence, RI, 2019,
  pp.~203--215. \MR{3996490}

\end{thebibliography}

\end{document}